\documentclass[reqno]{amsart}
\usepackage{amsfonts,amssymb,amsmath,amsthm}
\usepackage[utf8]{inputenc}
\usepackage{url, hyperref, calc, enumerate}
\usepackage[all]{xy}
\hypersetup{colorlinks   = true,
			citecolor    = green,
			linkcolor    = blue,
			urlcolor	 = cyan }
\usepackage{cleveref, caption, geometry, xcolor}
\newtheorem{theorem}{Theorem}[section]
\newtheorem{lemma}[theorem]{Lemma}

\theoremstyle{definition}
\newtheorem{definition}[theorem]{Definition}
\newtheorem{remark}[theorem]{Remark}
\newtheorem{example}[theorem]{Example}
\newtheorem{notation}{Notation}
\numberwithin{equation}{section}

\newcommand\numberthis{\addtocounter{equation}{1}\tag{\theequation}}
\newenvironment{acknowledgements} {\begin{abstract}} {\end{abstract}}

\global\long\def\C{\mathbb{C}}

\global\long\def\R{\mathbb{R}}

\global\long\def\Z{\mathbb{Z}}

\global\long\def\F{\mathbb{F}}

\global\long\def\mf#1{\mathfrak{#1}}
\global\long\def\mc#1{\mathcal{#1}}

\global\long\def\im{\mathrm{im}}

\global\long\def\.{,\dots ,}

\global\long\def\Aut{\operatorname{Aut}}

\global\long\def\so{\mathfrak{so}}
\global\long\def\Ad{\mathrm{Ad}}
\global\long\def\ad{\mathrm{ad}}

\global\long\def\su{\mathfrak{su}}
\global\long\def\so{\mathfrak{so}}

\global\long\def\g{\mathfrak{g}}
\global\long\def\h{\mathfrak{h}}

\global\long\def\m{\mathfrak{m}}

\global\long\def\Span{\operatorname{span}}

\global\long\def\End{\operatorname{End}}
\global\long\def\sd{\mathrm{d}}
\global\long\def\<{\langle}
\global\long\def\>{\rangle}

\author{Reinier Storm}
\address{KU Leuven, Department of Mathematics, Celestijnenlaan 200B -- Box 2400, BE-3001 Leuven, Belgium} 
\email{reinier.storm@kuleuven.be}


\begin{document}
	
\title{Lagrangian submanifolds of the nearly Kähler full flag manifold $F_{1,2}(\C^3)$}

\begin{abstract}
	In this article the framework created by Cartan to produce local differential invariants for submanifolds of homogeneous spaces is applied to classify all totally geodesic Lagrangian submanifolds and all homogeneous Lagrangian submanifolds of the nearly Kähler manifold of full flags in $\C^3$. 
\end{abstract}
\maketitle
	
\section{Introduction}
Nearly Kähler manifolds are one of the sixteen classes of almost hermitian manifolds, which are classified by Gray and Hervella in \cite{Gray1980}. The most common definition of a nearly Kähler almost complex structure $J$ in the literature is $(\nabla^g_XJ)(X) = 0$ for all vector fields $X$, where $\nabla^g$ denotes the Levi-Civita connection. A nearly Kähler non-Kähler manifold is generally referred to as a strict nearly Kähler manifold.
In \cite{Grunewald1990} a nice alternative characterization is given which states: a $6$-dimensional manifold with a non-parallel Killing spinor is strict nearly Kähler with respect to its natural almost complex structure. Conversely, a strict nearly Kähler manifold admits a non-parallel Killing spinor which induces the almost complex structure.
In \cite{Nagy2002} different types of nearly Kähler manifolds are distinguished and a classification result in terms of these types is proven. Moreover, it is shown in \cite{Nagy2002} that the Gray-Wolf conjecture ``any homogeneous strict nearly Kähler manifold is a $3$-symmetric space with its canonical almost complex structure'' is reduced to the $6$-dimensional case. 
The $6$-dimensional case of the Gray-Wolf conjecture is proven in \cite{Butruille2005}.
The $6$-dimensional homogeneous strict nearly Kähler manifolds are the round sphere $S^6 \cong G_2/SU(3)$ with its invariant almost complex structure, $S^3 \times S^3 = SU(2)\times SU(2)\times SU(2)/SU(2)_\Delta$ with its natural Riemannian $3$-symmetric structure, and the twistor spaces $F_{1,2}(\C^3)$ and $\C P^3$ over $\C P^2$ and $S^4$, respectively, with their natural non-integrable almost complex structures.
The first complete non-homogeneous strict nearly Kähler manifolds are discovered in \cite{Foscolo2017} on $S^3\times S^3$ and $S^6$.

A Lagrangian submanifold of an almost hermitian manifold is a submanifold such that the almost complex structure maps the tangent bundle of the submanifold bijectively to its normal bundle. In particular the submanifold needs to have half the dimension of the ambient space.
Ejiri was the first who noticed that a Lagrangian submanifold of the homogeneous nearly Kähler $S^6$ is automatically minimal and oriented, see \cite{Ejiri1981}. This was later generalized to all strict nearly Kähler manifolds in \cite{Gutowski2003}. Lagrangian submanifolds of the homogeneous nearly Kähler $S^6$ are by now very well studied, see for example \cite{Dillen1987,Dillen1990,Dillen1996, Lotay2011}.

More recently, Lagrangian submanifolds on the nearly Kähler manifold $S^3\times S^3$ were constructed from generalized Killing spinors on $S^3$ by Moroianu and Semmelmann in \cite{Moroianu2014}. In \cite{Dioos2018} all Lagrangian submanifolds with constant sectional curvature are classified by means of a triple of angle functions the authors introduce. This method is further studied in \cite{Bektas2018, Zhang2016}.
In \cite{Bektas2019} Lagrangian submanifolds of $S^3\times S^3$ are constructed from minimal surfaces in $S^3$ and conversely if the projection of a Lagrangian submanifold under a natural projection $S^3\times S^3 \to S^3$ is a $2$-dimensional surface, then this surface is minimal.

\subsection{Results}
The aim of this work is to classify all totally geodesic and all homogeneous Lagrangian submanifolds of the nearly Kähler manifold $(F_{1,2}(\C^3),g_{NK},J_{NK})$ up to the symmetry group of the ambient space. 
In view of Klein's Erlanger program it is natural to consider two submanifolds $\Sigma_1$ and $\Sigma_2$ of a homogeneous space $G/H$ with transformation group $G$ to be congruent when there exists an element $g\in G$ such that $L_g(\Sigma_1) = \Sigma_2$, where $L_g:G/H\to G/H$ denotes the left multiplication by $g$. The transformation group $G$ we use are all isometries of $(F_{1,2}(\C^3),g_{NK})$ which also preserve the nearly Kähler almost complex structure $J_{NK}$.
To achieve our goal we use the framework developed by Élie Cartan as exposed in his book \cite{Cartan1935}, see also \cite{Griffiths1974} or \cite{Ivey2003}.
This essentially yields a complete set of local invariants for these equivalence classes. It also provides necessary and sufficient conditions on the invariants for there to exist a submanifold with the prescribed set of invariants.

To fix our notations and conventions we will start in the preliminaries by recalling Cartan's method of moving frames. Then the homogeneous nearly Kähler structure on the space $F_{1,2}(\C^3)$ of full flags of $\C^3$ is defined. We finish by describing the twistor fibration $\pi:F_{1,2}(\C^3)\to \C P^2$ as a homogeneous fiber bundle.

In \Cref{sec:moving frame} a particular moving frame of Lagrangian submanifolds of $F_{1,2}(\C^3)$ is defined. This is then simplified for the special case that the image of the Lagrangian submanifold under the twistor fibration is $2$-dimensional. In \Cref{sec:second fund form} the second fundamental form is expressed in terms of this moving frame. For this the expression of the canonical connection is used, which is briefly described in \Cref{sec:appendix}.

In \Cref{sec:totally geodesic} all totally geodesic Lagrangian submanifolds of $F_{1,2}(\C^3)$ are classified, see \Cref{thm:totally geodesic lag.}. 
There are only two complete totally geodesic Lagrangian submanifolds of $F_{1,2}(\C^3)$ up to isometries of $F_{1,2}(\C^3)$. They are the real flag manifold $F_{1,2}(\R^3)\subset F_{1,2}(\C^3)$ for which the induced metric has constant sectional curvature and $S^3\subset F_{1,2}(\C^3)$ for which the metric is a Berger metric, see \Cref{ex:homogeneous Lagrangians F12} and \Cref{ex:homogeneous Lagrangians S3}, respectively.

In \Cref{thm:homogeneous Lag} all homogeneous Lagrangian submanifolds of $F_{1,2}(\C^3)$ are classified. Besides the complete totally geodesic Lagrangian submanifolds there is one extra non-totally geodesic homogeneous Lagrangian submanifold. It is diffeomorphic to $\R P^3$ and the induced metric is induced from a Berger metric on $S^3$ through the antipodal map, see \Cref{ex:hom RP3}.

\section{preliminaries}

\subsection{Moving frames}\label{sec:intro moving frames}
Let $G$ be a Lie group with Lie algebra $\g$ and left invariant Maurer-Cartan form $\omega_G$. An important part of Cartan's method of moving frames is the following theorem. For a proof one can also consult \cite{Griffiths1974} or \cite{Ivey2003}.
\begin{theorem}[Cartan]\label{thm:Cartan}
Let $i_1,i_2:\Sigma\to G$ be two smooth maps on a connected manifold $\Sigma$. Then $i_1^*\omega_G = i_2^*\omega_G$ if and only if there exists an element $g\in G$ such that
	\[
	i_1= L_g \circ i_2,
	\]
	where $L_g:G\to G$ denotes the left multiplication by $g$.
	Moreover, given a smooth $\g$-valued $1$-form $\alpha:T\Sigma\to \g$ on  a simply connected manifold $\Sigma$ which satisfies the Maurer-Cartan equation
	\[
	\sd\alpha + \frac{1}{2}[\alpha,\alpha] = 0,
	\]
	then there exists a smooth map $i:\Sigma\to G$ such that $i^*\omega_G = \alpha$.
\end{theorem}
In the case of submanifolds of a Lie group this theorem completely solves the congruence problem and the existence problem. If the ambient space is a homogeneous space $G/H$ with $H\neq \{e\}$, then one needs an extra ingredient. 
To keep the exposition simpler we will now only illustrate this procedure for Lagrangian submanifolds.

Let $G/H$ be some Hermitian homogeneous space with an invariant almost complex structure $J$ and compatible metric $g$. Let $\omega_{J} := g(J-,-)$ be the almost symplectic form. 

\begin{definition}
	Let $\dim(G/H) =2n$ and let $\iota:L^n\to G/H$ be the inclusion of a submanifold. The submanifold $\iota(L^n)\subset G/H$ is \emph{Lagrangian} if $\iota^*(\omega_{J}) = 0$.
\end{definition}

\begin{remark}
	The above definition of a Lagrangian submanifold is in the literature also called a totally real submanifold of half the dimension of the ambient space.
\end{remark}
Assume we have a reductive decomposition $\g=\h\oplus \m$, where $\g$ is the Lie algebra of $G$ and $\h$ is the Lie algebra of $H$. The subspace $\m$ is naturally identified with $T_{eH}(G/H)$ by $X \mapsto \left.\frac{d}{dt}\right|_{t=0} e^{tX}\cdot H$ for $X\in \m$.
The group $G$ is interpreted as a reduction of the unitary frame bundle $G\subset \mathcal{F}_{U}(G/H)$. This inclusion is defined by $g\mapsto \sd L_g$, where $\sd L_g:\m\cong T_{eH}(G/H)\to T_{gH}(G/H)$ is a unitary framing of $T_{gH}(G/H)$.
By a (local) moving frame along $L^n$ we will always mean a (local) section of $\iota^*\mathcal{F}_{U}(G/H)$ which takes values in $G$, i.e. a function $\tilde{\iota}:L^n\to G$ such that the diagram
\[
\xymatrix{
	& G\ar[d]^\pi\\
	L^n\ar[r]^\iota \ar[ur]^{\tilde{\iota}} & G/H
}
\]
commutes. From now on a moving frame or a lift is always assumed to be local even when this is not mentioned.
For any two lifts $\tilde{\iota}_1:L^n\to G$ and $\tilde{\iota}_2:L^n\to G$ there exists a function $\psi:L^n\to H$ such that
\[
\tilde{\iota}_1(x)\cdot \psi(x) = \tilde{\iota}_2(x),\quad \forall x\in L^n
\]
holds. 
Let $\omega_i := \tilde{\iota}_i^* \omega_G$ denote the pullbacks of the Maurer-Cartan form $\omega_G$ of $G$ for $i=1,2$. Then
\[
\omega_2 = \Ad(\psi)^{-1}\omega_1 + \omega_G(\sd\psi)
\]
holds. In particular the pullback of the Maurer-Cartan form depends on a choice of lift. Thus in order to use \Cref{thm:Cartan} to decide whether two submanifolds of $G/H$ are congruent we need a way to define a unique lift. 
Below we roughly describe the general procedure how to find such a unique lift. This is then applied in \Cref{sec:moving frame} to Lagrangian submanifolds of the nearly Kähler manifold $F_{1,2}(\C^3)$.

First let $\omega_\h$ and $\omega_\m$ define the $\h$ and $\m$ component of the Maurer-Cartan form. 
We can identify the almost symplectic form $\omega_{J}$ with an element of $\Lambda^2 \m$ which is invariant under the action of the isotropy group $H$. This $2$-form in $\Lambda^2\m$ is also denoted by $\omega_J$.
A submanifold $\iota:L^n \to G/H$ is Lagrangian if and only if for a lift $\tilde{\iota}$ the image $(\tilde{\iota}^*\omega_\m)(T_x L^n) \subset \m$ is a linear Lagrangian subspace with respect to $\omega_J\in \Lambda^2\m$ for all $x\in L^n$. For a different lift related to $\tilde{\iota}$ by $\psi:L^n\to H$ this linear Lagrangian subspace for $x\in L$ is given by $\Ad(\psi(x))^{-1}\cdot (\tilde{\iota}^*\omega_\m)(T_x L^n)$. Remember that $H\subset U(\m)$ and thus $H$ acts on the space of linear Lagrangian subspaces of $\m$ which is itself a homogeneous space given by $\mathrm{Lag}(\m):=U(\m)/SO(\m)$. 
For our purpose it suffices to find a smallest dimensional parametrized submanifold $S\subset \mathrm{Lag}(\m)$ such that $H\cdot S = \mathrm{Lag}(\m)$.
The submanifold $S$ is stratified into the different orbit types of the $H$-action.
In Cartan's method of moving frames one demands that the principal orbit type $S_0\subset S$ is a slice for the $H$-action. The submanifold $S$ presented in \Cref{sec:moving frame} has this property.
For a point $\mf l\in S_0$ the stabilizer subgroup $H_0\subset H$ of $\mf l$ preserves $S_0$ by assumption. If $H_0 = \{e\}$, then we have found a unique moving frame for Lagrangian submanifolds which satisfies $(\tilde{\iota}^*\omega_\m)(T_xL)\subset S_0$ for all $x\in L$. 
If $H_0\neq \{e\}$, then we decompose $\h = \h_0 \oplus \m_0$, were $\m_0$ is a reductive complement for the $H_0$-action and we can again look for a slice and repeat the same procedure as above for $(\tilde{\iota}^*\omega_{\m_0})(T_x L^n)$. If $(\tilde{\iota}^*\omega_{\m_0})(T_x L^n) = \{0\}$ we can not proceed and no unique lift is found. It is possible that these steps have to be repeated a number of times.
If a unique frame is found, then the congruence problem and equivalence problem for submanifolds are solved by applying \Cref{thm:Cartan} in the unique moving frame.
For the singular strata the stabilizer subgroup in $H$ is larger. The above procedure can in principle also be followed in this case.

\subsection{The flag manifold \texorpdfstring{$F_{1,2}(\C^3)$}{F(1,2)}\label{sec:the flag manifold}}
We will start by describing a naturally reductive structure together with the nearly Kähler structure on $\F_{1,2}(\C^3)$ which can be represented as the homogeneous space $SU(3)/U(1)^2$. We pick the following basis of $\su(3)$: 
\[
\begin{array}{c c c c}
h_1 := \begin{pmatrix}
-I & 0 & 0\\
0 & 0 & 0\\
0 & 0 & I
\end{pmatrix}, & h_2 := \begin{pmatrix}
\frac{I}{\sqrt{3}} & 0 & 0\\
0 & \frac{-2I}{\sqrt{3}} & 0\\
0 & 0 & \frac{I}{\sqrt{3}}
\end{pmatrix},& m_1 := \begin{pmatrix}
0 & -1 & 0\\
1 & 0 & 0\\
0 & 0 & 0
\end{pmatrix},& m_2 := \begin{pmatrix}
0 & 0 & 0\\
0 & 0 & -1\\
0 & 1 & 0
\end{pmatrix},\\
m_3 :=\begin{pmatrix}
0 & 0 & -1\\
0 & 0 & 0\\
1 & 0 & 0
\end{pmatrix}, & m_4 := \begin{pmatrix}
0 & I & 0\\
I & 0 & 0\\
0 & 0 & 0
\end{pmatrix}, & m_5 :=\begin{pmatrix}
0 & 0 & 0\\
0 & 0 & I\\
0 & I & 0
\end{pmatrix}, & m_6 :=\begin{pmatrix}
0 & 0 & I\\
0 & 0 & 0\\
I & 0 & 0
\end{pmatrix}.
\end{array}
\]
This basis is orthonormal with respect to $\frac{-1}{12}B_{\su(3)}$, where $B_{\su(3)}$ is the Killing form of $\su(3)$. Let $\h:=\Span\{h_1,h_2\}$ and $\m:=\Span\{m_1,m_2,m_3,m_4,m_5,m_6\}$. The decomposition $\su(3)=\h\oplus \m$ defines a naturally reductive structure, where $m_1,\dots,m_6$ form an orthonormal basis for the naturally reductive metric $g$. Let $\m_1 = \mbox{span}\{m_1,m_4\}$, $\m_2 = \mbox{span}\{m_2,m_5\}$ and $\m_3 = \mbox{span}\{m_3,m_6\}$. Then $\m = \m_1 \oplus \m_2\oplus \m_3$ is a decomposition into irreducible $U(1)^2$-modules. 
The curvature $R$ and the torsion $T$ of the naturally reductive connection correspond to the tensors
\[
R(x,y) = - [x,y]_\h  \quad \mathrm{and}\quad T(x,y) = -[x,y]_\m = m_{123} + m_{156} - m_{246} -m_{345},
\]
where $m_{ijk}$ denotes $m_i\wedge m_j\wedge m_k$.
The dual basis of $\g$ will be denoted by $h^1,h^2,m^1,m^2,m^3,m^4,m^5,m^6$. Many of the computations that will be done rely on the exterior derivatives of this dual basis. These are given by:
\begin{align*}
\sd h^1 &= -m^1\wedge m^4 - m^2\wedge m^5 - 2m^3\wedge m^6, \\
\sd h^2 &=  \sqrt{3}m^1\wedge m^4 - \sqrt{3}m^2\wedge m^5, \\
\sd m^1 &= h^1\wedge m^4 -\sqrt{3}h^2\wedge m^4 + m^2 \wedge m^3 + m^5\wedge m^6, \\
\sd m^2 &= h^1\wedge m^5 + \sqrt{3}h^2\wedge m^5 - m^1\wedge m^3 - m^4\wedge m^6, \numberthis \label{eq:ext deriv} \\
\sd m^3 &= 2h^1\wedge m^6 + m^1\wedge m^2 -m^4\wedge m^5, \\
\sd m^4 &= -h^1\wedge m^1 +\sqrt{3}h^2\wedge m^1 + m^2\wedge m^6 + m^3\wedge m^5, \\
\sd m^5 &= -h^1\wedge m^2 -\sqrt{3}h^2 \wedge m^2 - m^1\wedge m^6 - m^3\wedge m^4,  \\
\sd m^6 &= -2h^1\wedge m^3 + m^1\wedge m^5 - m^2\wedge m^4. 
\end{align*}
The invariant $2$-forms in $\Lambda^2\m^*$ under the isotropy representation are spanned by $m^1\wedge m^4,\ m^2\wedge m^5$ and $m^3\wedge m^6$. Thus there are eight $SU(3)$-invariant almost symplectic structures and they are of the form $\pm m^1\wedge m^4 \pm m^2\wedge m^5 \pm m^3\wedge m^6$. Let $\Aut(G,H)$ be the group of automorphisms of $G$ which preserve $H$. The group of all signed permutation matrices in $SO(3)\subset SU(3)$ is contained in $\Aut(G,H)$ through the adjoint representation. Under this action there are only two orbits of almost symplectic forms. They are represented by 
\[
\omega_{NK} := m^1\wedge m^4 + m^2\wedge m^5 - m^3 \wedge m^6,
\]
and
\begin{equation}\label{eq:kahler form}
\omega_K := m^1\wedge m^4 + m^2\wedge m^5 + m^3 \wedge m^6.
\end{equation}
The corresponding almost complex structure $J$ is defined by the equation $g(J(x),y)=\omega(x,y)$ for all $x,y\in \m$. The two almost complex structures $J_{NK}$ and $J_K$ correspond to the nearly Kähler and the Kähler structure on $F_{1,2}(\C^3)$, respectively. 
\begin{remark}
	The flag manifold $F_{1,2}(\C^3)$ also carries a Kähler-Einstein metric.  
	This metric $g_K$ is expressed in terms of the nearly Kähler metric $g$ as
	\[
	g_K = g|_{\m_1\oplus  \m_2} \oplus 2 g|_{\m_3}.
	\]
	A quick computation shows $\omega_K$ is parallel for the Levi-Civita connection corresponding to $g_K$.
	The corresponding symplectic form is the Kirillov-Kostant-Souriau symplectic form on the adjoint orbit $\Ad(SU(3))\cdot h_1 \cong F_{1,2}(\C^3)$.
\end{remark} 
We will focus on the nearly Kähler structure. From now on the almost complex structure $J_{NK}$ will simply be denoted by $J$ and the corresponding almost symplectic form $\omega_{NK}$ by $\omega_J$.
Let us define the complex $1$-forms $z^i := m^i + i J m^i$, so $z^1 = m^1 +im^4$, $z^2 = m^2 +im^5$ and $z^3 = m^3 -im^6$. A complex volume form on $\m$ is given by 
\begin{align*}
\Upsilon &:= i^3 z^1 \wedge  z^2 \wedge z^3 = i^3 (m^1+im^4)\wedge (m^2+im^5)\wedge (m^3-im^6)\\
&= m^{456} - m^{126} - m^{135} + m^{234} - i(m^{123} + m^{156} -m^{246} - m^{345}) = \mathrm{Re}(\Upsilon) + i \mathrm{Im}(\Upsilon).
\end{align*}
The complex volume form $\Upsilon$ is invariant under the isotropy representation or equivalently the holonomy representation of the naturally reductive connection. In other words the triple $(g,\omega_J,\Upsilon)$ equips $F_{1,2}(\C^3)$ with an $SU(3)$-structure.

The de Rham differential of the almost symplectic form is given by
\begin{equation}\label{eq: d w_J}
\sd \omega_J = -3(m^{126} + m^{135} - m^{234} -m^{456}) = 3 \mathrm{Re}(\Upsilon).
\end{equation}
We see that $\sd \omega_J$ is a non-degenerate $3$-form and in particular it is not closed. Furthermore, we have $\sd \mathrm{Im}(\Upsilon) = -2\omega\wedge \omega$.
By the definition of a nearly Kähler manifold in \cite{Foscolo2017} we conclude $(F_{1,2}(\C^3),g,J)$ is a nearly Kähler manifold. Alternatively a simple computation yields $(\nabla^g_X J)(X) = 0$ for all vector fields $X$.
%

\begin{remark}\label{rem:dwj = 0 and Tj=0}
	A Lagrangian submanifold $i:L\to  F_{1,2}(\C^3)$ of the nearly Kähler manifold $F_{1,2}(\C^3)$ is by definition a submanifold for which $i^*(\omega_J) =0$. This implies $ 0 = \sd i^*(\omega_J) = i^*(\sd\omega_J) = i^*(3\mathrm{Re}(\Upsilon))$ and thus $L$ is automatically calibrated by $\mathrm{Im}(\Upsilon)$ in the sense of generalized calibrations as studied in \cite{Gutowski2003}. For a Calabi-Yau manifold, such a submanifold is called a special Lagrangian submanifold. For non-integrable $SU(3)$-structures this terminology is less common in the literature.
	Define $T^J(x,y,z) := T(x,y,J(z)) = g([x,y]_\m,J(z))$. A computation shows the tensor $T^J$ to be totally skew-symmetric and given by $T^J = 3 \mathrm{Re}(\Upsilon)$. Thus the tensor $T^J$ vanishes on every Lagrangian submanifold of $F_{1,2}(\C^3)$. We will see in \Cref{sec:second fund form} that the vanishing of $T^J$ implies that the second fundamental form of a Lagrangian submanifold corresponds to a totally symmetric $3$-tensor on the Lagrangian.
\end{remark}
From \cite{Onishchik1992} we know the isometry group of $F_{1,2}(\C^3) \cong SU(3)/U(1)^2$ is given by $\mathrm{Iso}(F_{1,2}(\C^3)) \cong G\cdot \Aut(G,H)$, where $G=SU(3)$, $H=U(1)^2$, the product is taken in 
\[
\mathrm{Sim}_{G}(\F_{1,2}(\C^3)) := \{f\in \mathrm{Diff}(\F_{1,2}(\C^3)):f(gx) = \phi(g)f(x),~ g\in G,~ x\in \F_{1,2}(\C^3),~ \phi \in \Aut(G)\}
\]
and $\Aut(G,H)$ denotes the group of all automorphisms of $G$ which preserve $H$.
For the study of Lagrangian submanifolds we are interested in the symmetry group of $(F_{1,2}(\C^3),g,J)$, i.e. all isometries which also preserve the almost complex structure. The group of symmetries of $(F_{1,2}(\C^3),g,J)$ is given by $\mathrm{Inn}(G,H)\cdot G$, where $\mathrm{Inn}(G,H)\subset \Aut(G,H)$ is the subgroup of all inner automorphisms of $G$ which preserve $H$. 
Let $D\subset SO(3)$ denote the group of signed permutation matrices inside $SO(3)$.
The isotropy subgroup of the symmetry group is given by $D\cdot U(1)^2$.
Let $\rho:D\cdot U(1)^2\to SO(\m)$ denote the isotropy representation for the symmetry group of $(\F_{1,2}(\C^3,g,J)$.
An element of $U(1)^2$ is of the form
\[
A_{x,y} = \begin{pmatrix}
e^{i(y -x)} & 0 & 0\\
0 & e^{-2iy} & 0\\
0 & 0 & e^{i(x+y)}\\
\end{pmatrix}\in SU(3).
\]
In the basis $(m_1,m_2,m_3,m_4,m_5,m_6)$ we have
\begin{equation}\label{eq:iso rep}
\rho(A_{x,y}) = \begin{pmatrix}
\cos(x-3y) & 0 & 0 & -\sin(x-3y) & 0 & 0\\
0 & \cos(x+3y) & 0 & 0 & -\sin(x+3y) & 0\\
0 & 0 & \cos(2x) & 0 & 0 & -\sin(2x)\\
\sin(x-3y) & 0 & 0 & \cos(x-3y) & 0 & 0\\
0 & \sin(x+3y) & 0 & 0 & \cos(x+3y) & 0\\
0 & 0 & \sin(2x) & 0 & 0 & \cos(2x)\\
\end{pmatrix}.
\end{equation}
In this basis the isotropy representation of a signed permutation matrix $\sigma\in D$ is given by
\begin{equation}\label{eq:permutation rep.}
\rho(\sigma) = \sigma - J\sigma J,
\end{equation}
where $\sigma$ is viewed as a permutation matrix on $\mbox{span}\{m_1,m_2,m_3\}$. 
Note that the symmetries of $(F_{1,2}(\C^3),g,J)$ automatically preserve the complex volume form $\Upsilon$.
\begin{remark}
	On a $k$-symmetric space a tensor field $\delta$ of type $(1,1)$ is called \emph{canonical} if its value at the identity coset is a polynomial in $\delta$, where $\delta$ is the symmetry of order $k$. The most well known example of this is the one we are dealing with here. Namely, a $3$-symmetric space and the canonical almost complex structure which is given by $J = \frac{1}{\sqrt{3}}(\delta-\delta^2)$, see \cite{Wolf1968a, Wolf1968}. 
	For the flag manifold $F_{1,2}(\C^3)$ this $3$-symmetric structure $\delta\in \Aut(\su(3))$ is given by
	\[
	\delta = \Ad\left(\exp\left(\frac{2\pi}{3} h_1\right)\right).
	\]
	This satisfies $\delta^3 = 1$, $\ker(\delta-1) = \h$ and $\im(\delta-1) = \m$.
	The eigenvalues of $\delta$ are $1$, $e^{2\pi i/3}$ and $e^{-2\pi i/3}$.
\end{remark}

\subsection{Twistor fibration\label{sec:twistor fibration}}
The nearly Kähler almost complex structure on $F_{1,2}(\C^3)$ is also well known in relation to the twistor fibration
\[
\pi:F_{1,2}(\C^3) \to \C P^2.
\]
We will briefly describe this twistor fibration below in our notation, which is not the conventional approach. Let $(M^4,g)$ be an oriented Riemannian $4$-manifold. The twistor bundle $J(M^4)$ of $M^4$ is defined as the bundle of all pointwise oriented hermitian almost complex structures on $M^4$, i.e. the fiber over a point $x\in M^4$ is given by
\[
J_x(M^4) = \{J\in SO(T_x M^4) : J^2= -1\}.
\]
The fiber is isomorphic to $SO(4)/U(2) \cong S^2$.
Alternatively, the twistor bundle can be defined as the associated bundle of the principal $SO(4)$-frame bundle $\mc{F}_{SO}(M^4)$ by
\begin{equation}
J(M^4) \cong \mc{F}_{SO}(M^4)\times_{SO(4)} SO(4)/U(2).
\end{equation}
Let
\[
TJ(M^4) = T^v J(M^4)\oplus T^h J(M^4)
\]
be the splitting of the tangent bundle into the vertical subbundle and a horizontal complement determined by the Levi-Civita connection of $M^4$. Two natural almost complex structures on $J(M^4)$ are defined by
\[
J^{\pm}_p = J_{S^2}\pm p \in \End(T_p J(M^4)),
\]
where $J_{S^2}$ is the natural almost complex structure on the fiber $S^2$ and $p\in J(M^4)$ is an almost complex structure on $T_{\pi(p)} M^4 \cong T^h_p J(M^4)$. In \cite{Atiyah1978} this twistor bundle is used to study instantons and it is shown that $J^+$ is integrable if and only if  $(M,g)$ is anti-self-dual. The almost complex structure $J^-$ is never integrable.

Now we will look at the map $\pi :F_{1,2}(\C^3)\to \C P^2$ in more detail. Consider the Lie algebra decomposition $\g = \h \oplus \m_1 \oplus \m_2 \oplus \m_3$ introduced in \Cref{sec:the flag manifold}. Note that $\mf u(2) \cong \h \oplus \m_i\subset \su(3)$ is a subalgebra for $i=1,2,3$ and a reductive complement is given by $\m_j \oplus \m_k$, where $\{i,j,k\} = \{1,2,3\}$. Let $\mf k := \h \oplus \m_3$ and let $U(2)\cong K\subset SU(3)$ be the connected subgroup of $SU(3)$ with Lie subalgebra $\mf k$. We obtain an equivariant Riemannian submersion
\[
\pi :F_{1,2}(\C^3) = SU(3)/U(1)^2 \longrightarrow SU(3)/K \cong \C P^2.
\]
The tangent spaces of the fibers of $\pi$ form the distribution $\underline{\m_3} := G\times_H \m_3 \subset G\times_H \m \cong TF_{1,2}(\C^3)$.
There are actually three different maps $\pi_i: F_{1,2}(\C^3) \to \C P^2$, where the tangent distribution of the fibers is $\underline{\m}_i :=G\times_H \m_i$. These three maps are conjugate to each other by an action of the cyclic permutation group $\mathbb{Z}_3\subset SU(3)$, i.e. $\pi_{\sigma(i)} = \pi_i \circ \tilde{\sigma}^{-1}$, where $\sigma$ is a cyclic permutation and $\tilde{\sigma}\in SU(3) \subset \mathrm{Iso}(F_{1,2}(\C^3))$ the corresponding isometry of $F_{1,2}(\C^3)$. Hence, four our purposes it suffices to only consider the map $\pi = \pi_3$.

Consider the base space $\C P^2 \cong SU(3)/K$. The reductive decomposition of $\C P^2$ is given by 
\begin{equation}\label{eq:red decomp cp2}
\su(3) = \mf k \oplus \m_1 \oplus \m_2,
\end{equation}
where the isotropy algebra is given by $\mf k := \h\oplus \m_3$. This reductive decomposition is symmetric, i.e. $[\m_1\oplus \m_2, \m_1 \oplus \m_2] \subset \mf k$. This describes $\C P^2$ as a symmetric space. 
Let $g$ denote the metric on $\m_1 \oplus \m_2$ for which $m_1,m_4,m_2,m_5$ form an orthonormal basis and also denote the corresponding invariant metric on $\C P^2$ by $g$.
The Kähler form of $\C P^2$ corresponds to 
\[
\omega_{\C P^2}  = \frac{1}{\sqrt{3}}\ad(h_2)|_{\m_1\oplus \m_2} = m^1\wedge m^4 - m^2 \wedge m^5\in \Lambda^2(\m_1 \oplus \m_2).
\]
This induces the Fubini-Study metric with constant holomorphic sectional curvature equal to $-4$ on $\C P^2$. Note that the almost complex structure on $\m_1 \oplus \m_2$ induced from the Kähler structure on $\C P^2$ is \emph{not} equal to the almost complex nearly Kähler structure restricted to $\m_1 \oplus \m_2$. 

The restriction of $J_{NK}$ and $J_K$ to $\m_1 \oplus \m_2$ agree. For now denote this restriction by $I\in \End(\m_1\oplus \m_2)\cong \End(T_{eK} \C P^2)$.
Pick a representative $g\in SU(3)$ for a point $gH\in F_{1,2}(\C^3)$. An almost complex structure on $T_{\pi(g)} \C P^2$ is given by
\[
\sd L_g \circ I \circ  \sd L_g^{-1} \in \End(T_{\pi(g)} \C P^2),
\]
which does not depend on the representative $g\in gH$ because $I$ is $H$-invariant.
Another point in the same fiber of $\pi(g)$ is of the form $gkH$ with $k\in K$ and induces the almost complex structure
\[
\sd L_g \circ (\Ad(k)\circ I )\circ  \sd L_g^{-1} \in \End(T_{\pi(g)} \C P^2).
\]
Varying the point in the fiber  clearly produces all skew-symmetric orientation preserving almost complex structures on $T_{\pi(g)} \C P^2$. The distribution $\underline{\m_1}\oplus \underline{\m_2}\subset TF_{1,2}(\C^3)$ is the horizontal distribution of the Levi-Civita connection of $\C P^2$.
Thus the two natural almost complex structures $J^+$ and $J^-$ on $F_{1,2}(\C^3)$ induced from the twistor bundle are described above as the nearly Kähler almost complex structure and the Kähler complex structure, respectively.

\section{Moving frames of Lagrangian submanifolds\label{sec:moving frame}}
In this section the method of moving frames is applied to describe Lagrangian submanifolds of the nearly Kähler full flag manifold $F_{1,2}(\C^3)$. We will start by presenting three homogeneous Lagrangian submanifolds. In \Cref{thm:homogeneous Lag} we prove these three examples exhaust all homogeneous Lagrangian submanifolds.
In the next subsection a moving frame is described for any Lagrangian submanifold. Then we see how this moving frame can be simplified under the extra condition that $\pi(L)\subset \C P^2$ is a $2$-dimensional submanifold. In the last subsection we conclude with some remarks about the differential invariants obtained from Cartan's method of moving frames.

\begin{example}\label{ex:homogeneous Lagrangians F12}
	Consider the standard inclusion $SO(3)\subset SU(3)$. Let 
	\[
	f_{\tt st} = (\C\times \{0\} \times \{0\},~\C\times \C \times \{0\} )
	\]
	denote the `standard flag' in $\C^3$. Note that the isotropy group of the natural action of $SU(3)$ on flags of $\C^3$ is precisely the subgroup $U(1)^2$ with Lie algebra $\h\subset \su(3)$.
	The elements of $SO(3)$ that fix $f_{\tt st}$ are the diagonal matrices. These form a subgroup $\Z_2\times \Z_2\subset SO(3)$. This yields a homogeneous submanifold
	\[
	F_{1,2}(\R^3) \cong SO(3)/\Z_2\times \Z_2 \cong SO(3)\cdot f_{\tt st} \subset F_{1,2}(\C^3),
	\]
	where $F_{1,2}(\R^3)$ denotes the space of full flags on $\R^3$.
	Its tangent space at the origin is given by the projection of $\so(3)\subset \su(3)$ onto $\m$. This projection is equal to $\mbox{span}\{m_1,m_2,m_3\}\subset \m$, which is a Lagrangian subspace. This implies the submanifold is Lagrangian, since it is homogeneous. From now on this submanifold will be denoted by $F_{1,2}(\R^3) \subset F_{1,2}(\C^3)$. Alternatively, we can think of this Lagrangian submanifold as the natural inclusion of $F_{1,2}(\R^3)$ into $F_{1,2}(\C^3)$. 
	The induced metric on $F_{1,2}(\R^3)$ comes from a biinvariant metric on $SO(3)$. Therefore, the induced metric on $F_{1,2}(\R^3)$ has constant sectional curvature. The sectional curvature can easily be computed and is equal to $-\frac{1}{4}$. 
	The image $\pi(F_{1,2}(\R^3))\subset \C P^2$ of the Lagrangian submanifold under the twistor fibration defines a $2$-dimensional homogeneous submanifold of $\C P^2$. In this case this is the standard inclusion of $\R P^2$ into $\C P^2$ corresponding to the standard inclusion of $\R^2$ into $\C^2$.
\end{example}

\begin{example}\label{ex:homogeneous Lagrangians S3}
	Consider the connected subgroup $SU(2)\subset SU(3)$ with its Lie algebra given by 
	\[
	\mathrm{span} \left\{m_1 + m_2,~ -m_4 + m_5,~ m_6 - \sqrt{3}h_2\right\}.
	\]
	Alternatively, we can define $SU(2)$ as the stabilizer group of the vector $\begin{pmatrix}
	1 & 0 & 1\\
	\end{pmatrix}^T\in \C^3$. There are no elements of $SU(2)$ which fix the plane $\C \times \C \times \{0\} \subset \C^3$. Thus we obtain a homogeneous submanifold
	\[
	S^3 \cong SU(2)\cdot f_{\tt st} \subset F_{1,2}(\C^3).
	\]
	We easily see that the tangent space of this submanifold at the standard flag is a linear Lagrangian subspace of $\m$. This implies the submanifold is Lagrangian, since it is homogeneous. From now on this submanifold will be denoted by $S^3\subset F_{1,2}(\C^3)$.
	The induced metric on $S^3\subset F_{1,2}(\C^3)$ is a Berger metric. This metric is obtained from a round sphere with sectional curvature $-1$ by rescaling the metric in the direction of the fundamental vector field corresponding to $m_6 - \sqrt{3}h_2$ by a factor $\frac{1}{4}$. 
	The image $\pi(S^3)\subset \C P^2$ of $S^3$ under the twistor fibration yields a complex submanifold $\C P^1\subset \C P^2$ corresponding to the inclusion $\C \times \{0\}\times \C \subset \C^3$.
\end{example}

\begin{example}\label{ex:hom RP3}
	Consider the subgroup $K\subset SU(3)$ with its Lie algebra given by
	\[
	 \mbox{span}\left\{\sqrt{2}h_1 + \frac{m_1 + m_2}{\sqrt{2}}, \quad \sqrt{2}h_2  +\frac{m_1 - m_2}{\sqrt{6} } + \sqrt{\frac{2}{3}}m_6,\quad \frac{-m_4+m_5}{\sqrt{3}}  +  \frac{1}{\sqrt{3}}m_3\right\}.
	\]
	Then $K$ is conjugate to the standard subgroup $SO(3)\subset SU(3)$, more explicitly $K = O\cdot SO(3) \cdot O^{-1}$, where the matrix $O$ is given by
	\[
	O := \begin{pmatrix}
	\frac{1}{\sqrt{6}} & \frac{\sqrt{2}i}{\sqrt{3}} & \frac{1}{\sqrt{6}}\\
	\frac{1}{\sqrt{2}} & 0 & \frac{-1}{\sqrt{2}}\\
	\frac{-i}{\sqrt{3}} & \frac{-1}{\sqrt{3}} & \frac{-i}{\sqrt{3}}\\
	\end{pmatrix}.
	\]
	As in \Cref{ex:homogeneous Lagrangians F12} and \Cref{ex:homogeneous Lagrangians S3} the submanifold $K\cdot f_{\tt st}$ is a Lagrangian submanifold. The stabilizer subgroup of $f_{\tt st}$ in $K$ is trivial. Thus
	\[
	\R P^3 \cong SO(3) \cong K\cdot f_{\tt st}.
	\]
	In the sequel we will denote this Lagrangian simply by $\R P^3$. The induced curvature on $\R P^3$ is obtained from a metric of constant sectional curvature $-\frac{3}{4}$ and rescaling the metric in the direction of the fundamental vector field corresponding to $\frac{-m_4+m_5}{\sqrt{3}}  +  \frac{1}{\sqrt{3}}m_3$ by a factor $3$. The metric on $\R P^3$ is thus induced from a Berger sphere metric through the antipodal map.
	The image of $\R P^3$ under the twistor fibration defines a $3$-dimensional submanifold $\R P^3\subset \C P^2$. It is easily seen that this is a homogeneous Hopf hypersurface. The principal curvature of the structure vector field is equal to $4\sqrt{2}$.
	The other principal curvatures are $-\sqrt{2}$ and $\frac{1}{\sqrt{2}}$. This isoparametric hypersurface appears in the survey article \cite[p.~260]{Niebergall1997} as a tube over the complex quadric $Q^1\subset \C P^2$ and also as a tube over $\R P^2 \subset \C P^2$. 
\end{example}

\subsection{Moving frame of a generic Lagrangian}
In this section we will follow the procedure described in \Cref{sec:intro moving frames} to find a particular moving frame for any Lagrangian $L\subset F_{1,2}(\C^3)$. First note that the projection of $T_xL$ onto the distribution $\underline{\m}_3 = SU(3)\times_{U(1)^2} \m_3\subset TF_{1,2}(\C^3)$ always has a non-trivial kernel. Let $e_1\in \Gamma(TL)$ be a unit vector in the kernel of this projection. Let
\[
v_1 = \cos(\varphi)m_1 + \sin(\varphi)m_2 \qquad \mathrm{and}\qquad v_2 = \cos(\varphi)m_4 + \sin(\varphi)m_5
\] 
for a certain angle function $\varphi$. First we choose a moving frame such that
\[
\omega_\m(e_1) = \cos(\alpha)v_1 + \sin(\alpha)v_2,
\]
for some angle function $\alpha$.
The action of the subgroup $ H_1\subset H$ with its Lie algebra spanned by $h_1$ acts on $\omega_\m(e_1)$ by rotating the $\alpha$-angle.

Let $(e_1,e_2,e_3)$ be an orthonormal frame of $TL$.
From $\omega_J(e_1,e_2) = 0$ and $\omega_J(e_1,e_3) = 0$ it follows that the projections of $\omega_\m(e_2)$ and $\omega_\m(e_3)$ onto $\m_1 \oplus \m_2$ have to take value in the orthogonal complement of $\mbox{span}\{\omega_\m(e_1),J(\omega_\m(e_1))\}=  \mbox{span}\{v_1,v_2\}$. An orthonormal basis of this orthogonal complement is given by
\[
v'_1 = -\sin(\varphi)m_1 + \cos(\varphi)m_2\qquad \mathrm{and}\qquad v'_2 = -\sin(\varphi)m_4 + \cos(\varphi)m_5.
\]
It follows that 
\begin{align*}
\omega_\m(e_1) &= \cos(\alpha)v_1 + \sin(\alpha)v_2,\\
\omega_\m(e_2) &= a_1 v'_1 + a_2 v'_2 + a_3 w_1,\\
\omega_\m(e_3) &= b_1 v'_1 + b_2 v'_2 + b_3 w_2,
\end{align*}
where $a_i,b_i\in \R$ for $i=1,2,3$ and $(w_1,w_2)$ is an orthonormal basis of $\m_3$. 
Consider the matrix
\[
A:=\begin{pmatrix}
a_1 & a_2\\
b_1 & b_2\\
\end{pmatrix}.
\]
An element $R(s)\in H_1$ acts on $v_1'$ and $v_2'$ by
\[
R(s)v_1' = \cos(s)v'_1 + \sin(s)v'_2 \qquad \mathrm{and}\qquad R(s)v'_2 = -\sin(s)v'_1 + \cos(s) v'_2.
\]
By letting $R(s)$ act on the Lagrangian its induced action on the matrix $A$ is given by $A\mapsto R(s)\cdot A$. Alternatively, we can also choose a different basis on $TL$ by the following substitution 
\[
\begin{pmatrix}
e_2\\
e_3\\
\end{pmatrix}\mapsto \begin{pmatrix}
\hat{e}_2\\
\hat{e}_3\\
\end{pmatrix} = R(t)\cdot \begin{pmatrix}
e_2\\
e_3\\
\end{pmatrix}.
\]
In this new basis the matrix $A$ is given by $A\cdot R(t)$.
Using the polar decomposition of matrices we choose $s$ and $t$ such that in the new moving frame $A$ is diagonal, i.e.
\[
A\mapsto \begin{pmatrix}
\lambda & 0\\
0 & \mu\\
\end{pmatrix},
\]
for some $\lambda,\mu\in \R$. This yields a moving frame of the form
\begin{align*}
\omega_\m(e_1) &= \cos(\alpha)v_1 + \sin(\alpha)v_2,\\
\omega_\m(e_2) &= \lambda v'_1 + a_3w_1,\\
\omega_\m(e_3) &= \mu v'_2 + b_3w_2.
\end{align*}
Here we make some slight abuse of notation, because the orthonormal basis $w_1$, $w_2$ of $\m_3$ and also the angle $\alpha$ may have changed in the process but we still denote them with the same symbol.
Since $1 = \|\omega_\m(e_2)\| = \sqrt{\lambda^2 + a_3^2}$ we have 
\[
\omega_\m(e_2) = \sin(\beta) v_1' + \cos(\beta)w_1,
\]
for some angle $\beta$. The following hold $\omega_J(v'_1,v'_2) = 1$ and $\omega_J(w_1,w_2) =-1$. From this we obtain
\[
0= \omega_J(e_2,e_3) = \lambda\mu - a_3b_3.
\]
This implies that up to a sign we find
\[
\omega_\m(e_3) = \cos(\beta)v_2' + \sin(\beta)w_2.
\]
In total we have obtained a moving frame for any Lagrangian submanifold $L\subset F_{1,2}(\C^3)$ of the following form
\begin{align*}
\omega_\m(e_1) &= \cos(\alpha)(\cos(\varphi)m_1 + \sin(\varphi)m_2) + \sin(\alpha)(\cos(\varphi)m_4 + \sin(\varphi)m_5),\\
\omega_\m(e_2) &= \sin(\beta)(-\sin(\varphi)m_1 + \cos(\varphi)m_2) + \cos(\beta)(\cos(\theta)m_3 + \sin(\theta)m_6),\\
\omega_\m(e_3) &= \cos(\beta)(-\sin(\varphi)m_4 + \cos(\varphi)m_5) + \sin(\beta)(-\sin(\theta)m_3 + \cos(\theta)m_6),
\end{align*}
for some extra angle function $\theta$. Let $L_{\alpha,\beta,\varphi,\theta}\in \mathrm{Lag}(\m)$ denote the above linear Lagrangian. The submanifold $S$ described in \Cref{sec:moving frame} is the image of the map $(\alpha,\beta,\varphi,\theta)\mapsto L_{\alpha,\beta,\varphi,\theta}\in \mathrm{Lag}(\m)$. 
\begin{notation}\label{not:pullback forms}
	Let $\iota:L\to F_{1,2}(\C^3)$ be the inclusion of some Lagrangian submanifold. In the sequel we will only work on some submanifold of $F_{1,2}(\C^3)$. For this reason we will denote the pullback $\tilde{\iota}^*\alpha$ of some differential form $\alpha\in \Omega^\bullet (G)$ by some moving frame $\tilde{\iota}$ simply by $\alpha$ as well.
	The exterior derivative commutes with pullbacks. Thus in this notation the equations \eqref{eq:ext deriv} directly apply to the pulled-back $1$-forms of $h^i$ and $m^j$ as well. 
\end{notation}
Now we compute the integrability condition $ 0 = \sd \omega_J$ which holds for a Lagrangian submanifold. Equation \eqref{eq: d w_J} gives us 
\[
0=-\frac{1}{3}\sd \omega_J = (m^{126} + m^{135} - m^{234} - m^{456})(e_1,e_2,e_3) = \cos(\alpha-\theta).
\]
Thus $\alpha = \theta +\frac{\pi}{2} \mod \pi$ and we can write our moving frame as:
\begin{align}
\omega_\m(e_1) &= -\sin(\theta)v_1 + \cos(\theta)v_2,\notag\\
\omega_\m(e_2) &= \sin(\beta)v'_1 + \cos(\beta)w, \label{eq:generic moving frame}\\
\omega_\m(e_3) &= \cos(\beta)v'_2 + \sin(\beta)w'\notag,
\end{align}
where $w := \cos(\theta)m_3 + \sin(\theta)m_6$ and $w':= -\sin(\theta)m_3 + \cos(\theta)m_6$. The dimension of the parametrized submanifold $S\subset SU(\m)/SO(\m)$ described by \eqref{eq:generic moving frame} is equal to $3 = \dim(SU(\m)/SO(\m)) - \dim(H)$. Thus $S$ has the smallest possible dimension which a slice of the $H$-action can have.

\subsection{Framing for Lagrangians with \texorpdfstring{$\dim(\pi(L^3))=2$}{dim(pi(L3))=2}\label{sec:framing 2-dim in cp2}} 
In this section we will simplify the above framing in case the image of the Lagrangian submanifold under the twistor fibration $\pi : F_{1,2}(\C^3) \to \C P^2$ is a $2$-dimensional surface. This means that the intersection of $T_xL$ with the fiber over $x$ is $1$-dimensional for every $x\in L$. These are the submanifolds of $F_{1,2}(\C^3)$ which are Lagrangian for the nearly Kähler as well as for the Kähler structure.

Let $L\subset F_{1,2}(\C^3)$ be a Lagrangian submanifold such that $\pi(L)\subset \C P^2$ is $2$-dimensional.
Then the projection $\pi_*:T_xL \to T_{\pi(x)}\C P^2$ has a $1$-dimensional kernel for every $x\in L$. Thus for the moving frame along $L$ there is one tangent vector, say $e_3$, contained in the fiber direction, i.e.
\[
\omega_\m(e_3) = -\sin(\theta)m_3 + \cos(\theta)m_6.
\]
The subspace $\m_1\oplus \m_2$ is a complex subspace of $\m$. Therefore, the projection of $T_xL$ onto $\m_1 \oplus \m_2$ is a Lagrangian subspace in $\m_1 \oplus\m_2$ with respect to the restricted almost complex structure.
Suppose the projection of $T_xL$ onto $\m_1$ is surjective. The case when the projection onto $\m_1$ is $1$-dimensional can be treated separately and corresponds to the case $\sin(\varphi)\cos(\varphi)=0$ in \eqref{eq:moving frame dim(pi(L))=2} below. We can pick a frame $(e_1,e_2,e_3)$ of the form
\begin{align*} 
\omega_\m(e_1) &= m_4 + a_1 m_5 + b_1 m_2,\\
\omega_\m(e_2) &= m_1 + a_2 m_5 + b_2 m_2,\\
\omega_\m(e_3) &= -\sin(\theta)m_3 + \cos(\theta)m_6,
\end{align*}
where $a_i,b_i\in \R$ for $i=1,2$. Now we need to know how the isotropy group $H = U(1)\times U(1)$ acts on $\m_1\oplus \m_2$. We pick two connected subgroups with Lie subalgebras spanned by $\sqrt{3}h_1 -h_2$ and $\sqrt{3}h_1 + h_2$, respectively. These are closed $U(1)$-subgroups. We will refer to them as the first $U(1)$-factor and the second $U(1)$-factor, respectively.
The first $U(1)$-factor rotates the $(m_1,m_4)$-plane and fixes the $(m_2,m_5)$-plane. The second $U(1)$-factor rotates the $(m_2,m_5)$-plane an fixes the $(m_1,m_4)$-plane. If the moving frame is changed by a rotation about $\theta_1$ in the $(m_1,m_4)$-plane and a rotation about $\theta_2$ in the $(m_2,m_5)$-plane, i.e. the moving frame is changed by $(R(\theta_1),R(\theta_2))\in H$, and simultaneously a new frame is picked by
\[
\begin{pmatrix}
e_1\\
e_2\\
\end{pmatrix}\mapsto \begin{pmatrix}
\hat{e}_1\\
\hat{e}_2\\
\end{pmatrix} = R(-\theta_1)\cdot \begin{pmatrix}
e_1\\
e_2\\
\end{pmatrix},
\]
then the coefficients $a_1,a_2,b_1$ and $b_2$ change by
\[
\begin{pmatrix}
a_1 & b_1\\
a_2 & b_2\\
\end{pmatrix} \mapsto R(-\theta_1)\begin{pmatrix}
a_1 & b_1\\
a_2 & b_2\\
\end{pmatrix}R(-\theta_2).
\]
In the process also $\theta$ changes, but as we did before we will still denote this changed parameter with the same symbol.
Thus we can diagonalize this matrix and get a moving frame of the form
\begin{align*}
\omega_\m(e_1) &= \cos(\varphi_1) m_4 + \sin(\varphi_1) m_5 ,\\
\omega_\m(e_2) &= \cos(\varphi_2) m_1 + \sin(\varphi_2) m_2,\\
\omega_\m(e_3) &= -\sin(\theta)m_3 + \cos(\theta)m_6,
\end{align*}
for certain angle functions $\varphi_1$ and $\varphi_2$.
Furthermore, we have $ 0 = \omega_J(e_1,e_2) = -\cos(\varphi_1)\cos(\varphi_2) - \sin(\varphi_1)\sin(\varphi_2) = -\cos(\varphi_1-\varphi_2)$.
Hence, $\varphi_2 \equiv \varphi_1 + \frac{\pi}{2} \mod \pi$. The integrability condition $0 = \sd \omega_J$ now implies $\sin(\theta) = 0$. All in all we obtain a moving frame of the form
\begin{align*}\label{eq:moving frame dim(pi(L))=2}
\omega_\m(e_1) &= \cos(\varphi) m_4 + \sin(\varphi) m_5,\\
\omega_\m(e_2) &= -\sin(\varphi) m_1 + \cos(\varphi) m_2,\numberthis\\
\omega_\m(e_3) &= m_6.
\end{align*}
This moving frame is a special case of the frame described in the previous section by putting $\beta = \frac{\pi}{2}$ and  $\theta=0$.

\begin{remark}\label{rem:slice}
	It is pointed out in \Cref{rem:dwj = 0 and Tj=0} that for every Lagrangian submanifold $L\subset F_{1,2}(\C^3)$ the linear subspace $\omega_\m(T_xL)\subset \m$ takes values in the space of special Lagrangian subspaces of $\m$, i.e. $SU(\m)/SO(\m)\subset \mathrm{Lag}(\m)$. 
	The image of $(\beta,\varphi,\theta)\mapsto L_{\theta,\varphi,\beta}$, where $L_{\theta,\varphi,\beta}$ is the linear Lagrangian spanned by \eqref{eq:generic moving frame}, is partitioned into different orbit types for the isotropy representation. 
	
	It is not hard to check that the orbit type of the principal orbits is the identity group $[e]$. Thus up to finding a fundamental domain in $S = \{L_{\theta,\varphi,\beta}\}_{\theta,\varphi,\beta} \subset SU(\m)/SO(\m)$ the angle functions are differential invariants. In this unique moving frame also $x\mapsto B_{\su(3)}(\omega(T_xL),h_i)$ are differential invariants for $i=1,2$, where $B_{\su(3)}$ is the Killing form of $\su(3)$.
	
	The orbit type of the Lagrangian submanifold from \Cref{ex:homogeneous Lagrangians F12} is $[D]$, where $D$ is the group of all signed permutation matrices in $SO(3)$. The isotropy representation of $D$ is described in \eqref{eq:permutation rep.}.

	Using the Lie algebra action of $\h$ on $SU(\m)/SO(\m)$ one can easily conclude that up to congruence the only linear Lagrangian whose orbit type is not-discrete is the tangent space at the identity coset of $S^3\subset F_{1,2}(\C^3)$ from \Cref{ex:homogeneous Lagrangians S3}.
	The identity component of the stabilizer of this linear Lagrangian is given by the $1$-dimensional closed subgroup $\{\exp(th_2):t\in \R\}\subset U(1)^2$. This will be useful in \Cref{thm:homogeneous Lag}.
	
	The orbit type of the linear Lagrangian spanned by the tangent space at the origin of \Cref{ex:hom RP3} is equal to $[\Z_3]$, where $\Z_3$ is generated by the $h\cdot \Ad(\sigma) \cdot h^{-1}\in \Aut(G,H)$, where $h = \exp(\frac{\pi}{2}h_1)$ and the permutation $\sigma$  is given by
	\[
	\sigma = \begin{pmatrix}
	0 & -1 & 0\\
	0 & 0 & 1\\
	-1 & 0 & 0\\
	\end{pmatrix}.
	\]
	Thus the homogeneous Lagrangian submanifolds from \Cref{ex:homogeneous Lagrangians F12}, \Cref{ex:homogeneous Lagrangians S3} and \Cref{ex:hom RP3} are also distinguished by their orbit types.
\end{remark}

\section{Second fundamental form \label{sec:second fund form}}
From the moving frame of \Cref{sec:moving frame} we can compute the second fundamental form of a Lagrangian submanifold, see \Cref{sec:appendix}. For this we will use the naturally reductive connection. In \Cref{rem:dwj = 0 and Tj=0} we saw that the $3$-form $T^J$ defined by $T^J(x,y,z) = T(x,y,J(z))$ vanishes when pulled-back to $L$. Let $\nabla^g$ denote the Levi-Civita connection and $\nabla$ the naturally reductive connection on $F_{1,2}(\C^3)$. Let $(e_1,e_2,e_3)$ denote a local orthonormal frame of $L$.
The components of the second fundamental form are given by
\begin{align*}
h_{ij}^k &= g(\nabla^g_{e_i}e_j , Je_k) = g(\nabla_{e_i}e_j,Je_k)  -\frac{1}{2}T(e_i,e_j,Je_k)=  g(\nabla_{e_i}e_j,Je_k).
\end{align*}
Thus the second fundamental form of $L$ induced from the naturally reductive connection is the same as the one from the Levi-Civita connection. Since $J$ is $\nabla$-parallel the following holds
\[
h_{ij}^k = g(\nabla_{e_i}e_j,Je_k) = -g(e_j,\nabla_{e_i}(Je_k)) = -g(e_j, J\nabla_{e_i}e_k) = g(\nabla_{e_i}e_k,Je_j) = h_{ik}^j.
\]
This combined with the symmetry of $h_{ij}^k$ in the $ij$-indices implies that $h_{ij}^k$ is totally symmetric.

Let $\omega_\h(e_j) = l_{1j}h_1 + l_{2j}h_2$ for $j=1,2,3$ and certain functions $l_{ij}$.
Computing the components of the second fundamental form by using the naturally reductive connection and \Cref{sec:appendix} gives the following
\begin{align*}
h_{k1}^1 &= g(\nabla_{e_k}e_1,Je_1) = e_k(\theta) + l_{1k} + \sqrt{3}l_{2k}(\sin(\varphi)^2 - \cos(\varphi)^2),\\
h_{k2}^2 &= g(\nabla_{e_k}e_2,Je_2) = -e_k(\theta)\cos(\beta)^2 + l_{1k}(\sin(\beta)^2-2\cos(\beta)^2) - \sqrt{3}l_{2k}\sin(\beta)^2(\sin(\varphi)^2-\cos(\varphi)^2),\\
h_{k3}^3 &= g(\nabla_{e_k}e_3,Je_3) = -e_k(\theta)\sin(\beta)^2 + l_{1k}(\cos(\beta)^2-2\sin(\beta)^2)  -\sqrt{3}l_{2k}\cos(\beta)^2(\sin(\varphi)^2-\cos(\varphi)^2).
\end{align*}
We immediately see from these equations that Lagrangians in $F_{1,2}(\C^3)$ are minimal submanifolds, which is a well known result for $6$-dimensional nearly Kähler manifolds, see \cite{Gutowski2003}. The other components are given by
\begin{align*}
h_{11}^2 &= e_1(\varphi)\cos(\theta)\sin(\beta) - 2\sqrt{3}l_{21}\sin(\theta)\sin(\beta)\sin(\varphi)\cos(\varphi),\\
h_{11}^3 &= e_1(\varphi)\sin(\theta)\cos(\beta) + 2 \sqrt{3}l_{21}\cos(\theta)\cos(\beta)\sin(\varphi)\cos(\varphi),\\
h_{22}^1 &= e_2(\varphi)\sin(\beta)\cos(\theta) -2\sqrt{3}l_{22}\sin(\theta)\sin(\beta)\sin(\varphi)\cos(\varphi),\\
h_{33}^1 &= e_3(\varphi)\cos(\beta)\sin(\theta) + 2\sqrt{3}l_{23}\cos(\theta)\cos(\beta)\sin(\varphi)\cos(\varphi),\\
h_{23}^1 &= e_2(\varphi)\cos(\beta)\sin(\theta) + 2\sqrt{3}l_{22}\cos(\theta)\cos(\beta)\sin(\varphi)\cos(\varphi),\\
h_{31}^2 &= e_3(\varphi)\cos(\theta)\sin(\beta) - 2\sqrt{3}l_{23}\sin(\theta)\sin(\beta)\sin(\varphi)\cos(\varphi),\\
h_{12}^3 &= -e_1(\beta),\qquad h_{22}^3 = -e_2(\beta), \qquad h_{33}^2 = e_3(\beta).
\end{align*}
The symmetry of the second fundamental form now implies certain relations between the functions $\theta,\varphi,\beta,l_{ij}$ which will be used below to classify all totally geodesic and all homogeneous Lagrangian submanifolds.

\subsection{Second fundamental form if \texorpdfstring{$\dim(\pi(L)) = 2$}{dim(pi(L))=2}\label{sec:sff dim = 2}} 
The simplified moving frame we found in \Cref{sec:framing 2-dim in cp2} is a special case of the general moving frame by putting $\beta = \frac{\pi}{2}$ and  $\theta=0$. Therefore, we can simply read off the second fundamental form from the formulas above by substituting $\beta = \frac{\pi}{2}$ and $\theta=0$. This gives:
\begin{align*}
h_{k1}^1 &= g(\nabla_{e_k} e_1,Je_1) = l_{1k} + \sqrt{3}l_{2k}(\sin(\varphi)^2-\cos(\varphi)^2),\\
h_{k2}^2 &= g(\nabla_{e_k} e_2,Je_2) = l_{1k} - \sqrt{3}l_{2k}(\sin(\varphi)^2-\cos(\varphi)^2),\\
h_{k3}^3 &= g(\nabla_{e_k}e_3,Je_3) =-2l_{1k},
\end{align*}
\[
\begin{array}{l l l}
h_{11}^2 =  e_1(\varphi), & h_{22}^3 = 0, & h_{11}^3 = 0,\\
h_{33}^2 = 0, & h_{22}^1 = e_2(\varphi), & h_{33}^1 = 0,\\
h_{12}^3 = 0, & h_{23}^1 =0, & h_{31}^2 = e_3(\varphi). 
\end{array}
\]
The following short computations simplify this case even further.
The above equations immediately imply $l_{11} = l_{12} = 0$.
Thus the $1$-form $h^1$ simplifies to $h^1 = l_{13} e^3 = l_{13} m^6$. The exterior derivative of this equation yields
\begin{align*}
0&=\sd h^1 - \sd l_{13}\wedge m^6 - l_{13}\sd m^6 \\
&= -m^1\wedge m^4 - m^2\wedge m^5 - 2m^3\wedge m^6 - \sd l_{13}\wedge m^6 -l_{13}(2h^1\wedge m^3 + m^1\wedge m^5 - m^2 \wedge m^4)\\
&= - m^1\wedge m^4 - m^2\wedge m^5-l_{13}m^1\wedge m^5 + l_{13} m^2\wedge m^4 - \sd l_{13}\wedge m^6 \\
&= -l_{13}m^1\wedge m^5 + l_{13} m^2\wedge m^4 - \sd l_{13}\wedge m^6.
\end{align*}
Contracting this with $e_1\wedge e_2$ gives
\[
0 = -l_{13}\cos(\varphi)^2 - l_{13}\sin(\varphi)^2 = - l_{13}.
\]
We conclude $h^1=0$ holds. Furthermore, we have
\[
0 = h_{22}^3 = h_{32}^2 = -\sqrt{3}l_{23}(\sin(\varphi)^2 - \cos(\varphi)^2).
\]
Thus  $l_{23} = 0$ or $\sin(\varphi)^2 - \cos(\varphi)^2 = 0$. These simplifications will be used in the next sections.

\section{Classification of all totally geodesic Lagrangians\label{sec:totally geodesic}}
In this section we will classify all totally geodesic Lagrangian submanifolds of the homogeneous nearly Kähler manifold $F_{1,2}(\C^3)$. In \Cref{sec:second fund form} the second fundamental form was computed in the moving frame described in \Cref{sec:moving frame}. For a totally geodesic submanifold the second fundamental form vanishes. This gives a lot of extra equations on the functions $\theta,\varphi,\beta,l_{ij}$.

\begin{theorem}\label{thm:totally geodesic lag.} 
	Every totally geodesic Lagrangian is congruent to an open subset of one of the two homogeneous Lagrangian submanifolds, $F_{1,2}(\R^3)\subset F_{1,2}(\C^3)$ or $S^3\subset F_{1,2}(\C^3)$, which are described in \Cref{ex:homogeneous Lagrangians F12} and \Cref{ex:homogeneous Lagrangians S3}, respectively.
\end{theorem}
\begin{proof}
First we consider all totally geodesic Lagrangian submanifolds which satisfy the extra condition $\dim(\pi(L))=2$. Remember, for this case we found a moving frame of the form
\begin{align*}
\omega(e_1) &= l_{11}h_1 + l_{21}h_2 + \cos(\varphi) m_4 + \sin(\varphi) m_5,\\
\omega(e_2) &= l_{12}h_1 + l_{22}h_2 -\sin(\varphi) m_1 + \cos(\varphi) m_2,\\
\omega(e_3) &= l_{13}h_1 + l_{23}h_2 +m_6.
\end{align*}
From the discussion at the end of \Cref{sec:sff dim = 2} we know $h^1 = 0$ and $l_{23}(\sin(\varphi)^2 - \cos(\varphi)^2)=0$. Now we distinguish two cases: $\sin(\varphi)^2 - \cos(\varphi)^2=0 $ and $\sin(\varphi)^2 - \cos(\varphi)^2\neq0 $.

\paragraph{\textbf{Case $\sin(\varphi)^2 - \cos(\varphi)^2=0$:}}
First suppose $\sin(\varphi) = - \cos(\varphi)$. Up to changing an overall sign the moving frame takes the form:
\[
\omega_\m(e_1) = \frac{1}{\sqrt{2}}(m_1 + m_2),\quad \omega_\m(e_2) = \frac{1}{\sqrt{2}}(-m_4 + m_5),\quad \omega_\m(e_3) = m_6.
\]
Note that $\mbox{span}\{\omega_\m(e_1),\omega_\m(e_2),\omega_\m(e_3)\}$ is preserved by $\ad(h_2)$.
Furthermore, we have
\[
\begin{array}{l l l}
[h_2,\omega_\m(e_1)] = \sqrt{3}\omega_\m(e_2), & [h_2,\omega_\m(e_2)] = -\sqrt{3}\omega_\m(e_1),  & [h_2,\omega_\m(e_3)] = 0,
\\ \relax
[\omega_\m(e_1),\omega_\m(e_2)] = -\omega_\m(e_6)  + \sqrt{3}h_2, & [\omega_\m(e_1),\omega_\m(e_3)] = \omega_\m(e_2),  & [\omega_\m(e_2),\omega_\m(e_3)] = -\omega_\m(e_1).
\end{array}
\]
Thus we find a subalgebra
\[
\su(2) \cong \mbox{span}\{e_1,e_2,e_3 - \sqrt{3}h_2\}\subset \mf u(2) \cong \mbox{span}\{h_2,e_1,e_2,e_3\} \subset \su(3).
\]
Consequently, the moving frame $\tilde{\iota}:L\to SU(3)$, which is a lift of the inclusion $\iota:L\to F_{1,2}(\C^3)$, takes values in the subgroup $U(2)\subset SU(3)$ with the above subalgebra $\mbox{span}\{h_2,e_1,e_2,e_3\}$. This means that
\begin{equation}\label{eq:t.g. homogeneous}
\iota(x) = \pi(\tilde{\iota}(x)) \in U(2)\cdot f_{\tt st} = SU(2)\cdot f_{\tt st}, 
\end{equation}
where $f_{\tt st}$ denotes the standard flag defined in \Cref{ex:homogeneous Lagrangians F12}. On the right-hand-side of \eqref{eq:t.g. homogeneous} we have the homogeneous Lagrangian submanifold from \Cref{ex:homogeneous Lagrangians S3}. Since $L$ and $SU(2)$ have the same dimension it follows that $L$ is congruent to an open subset of the homogeneous Lagrangian $S^3\subset F_{1,2}(\C^3)$ from \Cref{ex:homogeneous Lagrangians S3}. 

Suppose $\sin(\varphi) =  \cos(\varphi)$. Up to choosing an overall sign the moving frame now takes the form
\[
\omega_\m(e_1) = \frac{1}{\sqrt{2}}(-m_1 + m_2),\quad \omega_\m(e_2) = \frac{1}{\sqrt{2}}(m_4 + m_5),\quad \omega_\m(e_3) = m_6.
\]
This case is congruent to the previously described Lagrangian after we apply the automorphism
\[
\Ad\begin{pmatrix}
0 & 0 & -1\\
0 & 1 & 0\\
1 & 0 & 0\\
\end{pmatrix} \in \Aut(G,H). 
\]

\paragraph{\textbf{Case $\sin(\varphi)^2-\cos(\varphi)^2\neq 0$:}}
Now the equations
\[
0 = h_{k1}^1 =  l_{1k} + \sqrt{3}l_{2k}(\sin(\varphi)^2-\cos(\varphi)^2) = \sqrt{3}l_{2k}(\sin(\varphi)^2-\cos(\varphi)^2)
\]
for $k=1,2,3$ imply $l_{2k} = 0$ or equivalently the $1$-form $h^2 = l_{21}e^1 + l_{22}e^2 + l_{23}e^3$ vanishes, where $e^i\in \Gamma(T^*L)$ is the dual basis of $e_i$. Taking the exterior derivative of this implies that $\sd h^2 = \sqrt{3}(m^1\wedge m^4 - m^2 \wedge  m^5) =0$. We already mentioned in the beginning of \Cref{sec:framing 2-dim in cp2} that the vectors $\omega_\m(e_1)$ and $\omega_\m(e_2)$ span a Lagrangian subspace in $\m_1 \oplus \m_2$ or equivalently the $2$-form $m^1\wedge m^4 + m^2\wedge m^5$ vanishes. Therefore, we find $m^1\wedge m^4 = 0$ and $m^2\wedge m^5 =0$. This holds if and only if $\cos(\varphi)\sin(\varphi) =0$. Up to choosing the signs of $\omega_\m(e_1)$, $\omega_\m(e_2)$ and $\omega_\m(e_3)$ the moving frame now takes the form
\[
\omega_\m(e_1) = m_4,\quad \omega_\m(e_2) = m_2,\quad \omega_\m(e_3) = m_6
\]
or
\[
\omega_\m(e_1) = m_5,\quad \omega_\m(e_2) = -m_1,\quad \omega_\m(e_3) = m_6.
\]
For both cases $\mbox{span}\{\omega_\m(e_1),\omega_\m(e_2),\omega_\m(e_3)\}$ is conjugate to $\mbox{span}\{m_1,m_2,m_3\}$ by the action of the isotropy representation. Consequently, we find a moving frame such that $\omega(T_x L) = \mbox{span}\{m_1,m_2,m_3\}$ for all $x\in L$. 
Just as before we conclude that the Lagrangian is congruent to an open subset of the homogeneous Lagrangian submanifold $F_{1,2}(\R^3)\subset F_{1,2}(\C^3)$ described in \Cref{ex:homogeneous Lagrangians F12}.

Consider an arbitrary totally geodesic Lagrangian submanifold. If $\sin(\beta)=0$ or $\cos(\beta)=0$, then the projection of $\omega_\m(T_xL)$ onto $\m_3$ is $1$-dimensional and thus we are in the above case. Similarly if $\sin(\varphi)=0$ or $\cos(\varphi) = 0$, then the projection onto $\m_1$ or $\m_2$, respectively, is $1$-dimensional. Since $\m_1,\m_2$ and $\m_3$ are conjugate by an element of $\Aut(G,H)$, see \Cref{sec:twistor fibration}, this also reduces to the above case.
Thus we assume $\sin(\beta),\cos(\beta), \sin(\varphi)$ and $\cos(\varphi)$ to be all non-zero.
From the equations $h_{11}^2 = h_{11}^3 =0$ it follows
\[
0 = -h_{11}^2\frac{\sin(\theta)}{\sin(\beta)} + h_{11}^3\frac{\cos(\theta)}{\cos(\beta)} =2\sqrt{3}l_{21}\sin(\varphi)\cos(\varphi).
\]
We can conclude $l_{21} = 0$. Similarly from $h_{22}^1 = h_{23}^1 = 0$ and $h_{33}^1 = h_{31}^2 =0$ we obtain $l_{22} = 0$ and $l_{23} = 0$, respectively. From this it follows that $0 = \sd h^2 = \sqrt{3}(m^1 \wedge m^4 -m^2 \wedge m^5)$ holds.
Contracting this with $e_1\wedge e_2$ yields $ 0 = 2\sin(\beta)\cos(\theta)\sin(\varphi)\cos(\varphi)$. Thus $\cos(\theta)=0$ and in particular $\theta$ is constant.
Combining this with $0 = h_{k1}^1$ implies $l_{1k} = 0$ for $k=1,2,3$. This implies that the $1$-form $h^1$ vanishes and thus also its exterior derivative vanishes. This gives
\[
0 = \sd h^1 = -m^1\wedge m^4 - m^2\wedge m^5 -2 m^3\wedge  m^6 = -\omega_J - 3m^3\wedge m^6.
\]
Consequently, the $2$-form $m^3\wedge m^6$ vanishes on $L$. This means the projection from $\omega_\m(T_x L)$ onto $\m_3$ is not surjective, which is a contradiction with our assumption that $\sin(\beta)\cos(\beta) \neq 0$.
\end{proof}

\section{Classification of all homogeneous Lagrangians\label{sec:homogeneous lag}}
In this section the homogeneous Lagrangian submanifolds of $F_{1,2}(\C^3)$ are classified. Recall that a submanifold $L\subset F_{1,2}(\C^3)$ is homogeneous when it is the orbit of a subgroup $K$ of the isometry group of $F_{1,2}(\C^3)$.
\begin{theorem}\label{thm:homogeneous Lag}
	A homogeneous Lagrangian submanifold of $F_{1,2}(\C^3)$ is either congruent to a complete totally geodesic Lagrangian submanifold, classified in \Cref{thm:totally geodesic lag.}, or it is congruent to $\R P^3\subset F_{1,2}(\C^3)$ described in \Cref{ex:hom RP3}.
\end{theorem}
\begin{proof}
	First note that a homogeneous Lagrangian submanifold is necessarily a complete submanifold. Without loss of generality we assume from now on that the identity coset is contained in $L$.

For a homogeneous Lagrangian submanifold $K/K_0\subset F_{1,2}(\C^3)$ the subalgebra $\mf k\subset \su(3)$ of $K\subset SU(3)$ can be either $\su(2)$ or $\mf u(2)$. First consider the case $\mf k \cong \mf u(2)$. This means the tangent space at the identity coset is a Lagrangian subspace of $\m$ which is preserved by the $1$-dimensional subalgebra $\mf k\cap \h$. 
From \Cref{rem:slice} it follows that if a Lagrangian subspace $\mf l$ is preserved by a $1$-dimensional subalgebra of $\h$, then this subalgebra has to be spanned by $h_2$ and $\mf l$ is conjugate by an element of $\Aut(G,H)$ to $\mbox{span}\{m_1 + m_2, -m_4 + m_5, m_6\}$. Together with $h_2$ this subspace spans a subalgebra of $\su(3)$ which is equal to the subalgebra in \Cref{ex:homogeneous Lagrangians S3}. Thus the homogeneous Lagrangian is congruent to $S^3\subset F_{1,2}(\C^3)$ from \Cref{ex:homogeneous Lagrangians S3}.

The remaining case is when the subalgebra $\mf k\subset \su(3)$ is equal to $\su(2)$. The Lagrangian submanifold is of the form $K/K_0=K\cdot H\subset G/H$, where $K_0\subset K$ is some discrete subgroup. In a small enough neighborhood of the identity coset there exists a lift $\tilde{\iota}:K/K_0\supset U \to K\subset G$ which is a local diffeomorphism. The left invariant vector fields on $\tilde{\iota}(U)$ determine a moving frame such that $\omega(e_i)\in \g$ is constant for $i=1,2,3$.
For this moving frame we thus obtain that $\mbox{span}\{\omega(e_1),\omega(e_2),\omega(e_3)\}\subset \g$ is a subalgebra and all the angles $\theta,\beta,\varphi$ and all the functions $l_{ij}$ are constant. 
If the projection of $TL$ onto one of the distributions $\underline{\m}_1$, $\underline{\m}_2$ or $\underline{\m}_3$ is not surjective, then it easily follows from the formulas of the second fundamental form in \Cref{sec:sff dim = 2} that the Lagrangian is totally geodesic. Consequently, the Lagrangian is classified by \Cref{thm:totally geodesic lag.}.
If $\sin(\beta)\cos(\beta) = 0$, then the projection of $TL$ onto $\underline{\m}_3$ is $1$-dimensional. If $\sin(\varphi)=0$ or $\cos(\varphi) = 0$, then the projection onto $\m_1$ or $\m_2$, respectively, is $1$-dimensional. Thus from now on we assume  $\sin(\beta)\cos(\beta)\sin(\varphi)\cos(\varphi) \neq 0$. 

First consider the case $\sin(\theta)\cos(\theta) \neq 0$.
The formulas in \Cref{sec:second fund form} of the second fundamental form immediately imply $h_{12}^3 = -e_1(\beta) = 0$ and thus we obtain 
\begin{align*}
0 &= h_{23}^1 = 2\sqrt{3}l_{22}\cos(\theta)\cos(\beta)\sin(\varphi)\cos(\varphi),\\
0 &= h_{31}^2 = -2\sqrt{3}l_{23}\sin(\theta)\sin(\beta)\sin(\varphi)\cos(\varphi).
\end{align*}
This implies $l_{22}=l_{23}=0$. Thus we have $h^2 = l_{21}e^1$ and $\sd h^2 = l_{21}\sd e^1$. By contracting both sides of $\sd h^2 = l_{21}\sd e^1$ with $e_1\wedge e_2$ it follows that
\begin{align*}
l_{21} \sd e^1(e_1\wedge e_2) &= -l_{21} e^1([\omega(e_1),\omega(e_2)]) = g([\omega(e_1),\omega(e_2)],-\sin(\theta)v_1 + \cos(\theta)v_2) \\
 &= -2\sqrt{3} (l_{21})^2 \cos(\theta)\sin(\beta)\sin(\varphi)\cos(\varphi)
\end{align*}
is equal to
\[
\sd h^2 (e_1\wedge e_2) = \sqrt{3}(m^1\wedge m^4 - m^2\wedge m^5)(e_1\wedge e_2) = 2\sqrt{3}\cos(\theta)\sin(\beta)\sin(\varphi)\cos(\varphi).
\]
This implies $\cos(\theta)= 0$, which contradicts the assumption $\sin(\theta)\cos(\theta)\neq 0$.

Consider the case $\sin(\theta)\cos(\theta) = 0$. For all four possibilities $\theta = 0$, $\theta=\frac{\pi}{2}$, $\theta=\pi$ and $\theta=\frac{3\pi}{2}$ the Lagrangians $\mbox{span}\{\omega_\m(e_1),\omega_\m(e_2),\omega_\m(e_3)\}$ are all conjugate to each other by either $\exp\left(\frac{\pi}{2}h_1\right)$, $\exp\left(\sqrt{3}\pi h_2\right)$ or a product of both.
Thus without loss of generality we put $\sin(\theta)=1$ and $\cos(\theta)=0$. 
From
\[
-2\sqrt{3} l_{23} \sin(\theta)\sin(\beta)\sin(\varphi)\cos(\varphi) =h_{31}^2 = h_{12}^3 = -e_1(\beta) =0
\]
it follows that $l_{23}=0$. Furthermore, we have $0 = h_{11}^3 = h_{31}^1 = l_{13} =0$. The moving frame of $L$ is thus of the following form
\begin{align*}
\omega(e_1) &= l_{11}h_1 + l_{21}h_2 -v_1,\\
\omega(e_2) &= l_{12}h_1 + l_{22}h_2 + \sin(\beta)v_1' + \cos(\beta)m_6,\\
\omega(e_3) &= \cos(\beta)v_2' - \sin(\beta)m_3.
\end{align*}
Also the following equations hold
\begin{align}\label{eq:thm homogeneous}
 0=h_{33}^2 &= l_{12}(\cos(\beta)^2-2\sin(\beta)^2) - \sqrt{3}l_{22}\cos(\beta)^2(\sin(\varphi)^2 - \cos(\varphi)^2),\\
 0=h_{33}^1 &= l_{11}(\cos(\beta)^2-2\sin(\beta)^2) - \sqrt{3}l_{21}\cos(\beta)^2(\sin(\varphi)^2 - \cos(\varphi)^2).\nonumber
\end{align}
First consider the case $\cos(\beta)^2 - 2\sin(\beta)^2 \neq 0$. Put $c = \sqrt{3}\cos(\beta)^2(\sin(\varphi)^2 - \cos(\varphi)^2 ) / (\cos(\beta)^2 -2\sin(\beta)^2)$, then we have
\[
l_{12} = c\cdot l_{22} \qquad \mathrm{and}\qquad l_{11} = c\cdot l_{21}.
\]
If $c=0$, then $h^1 = 0$ and just as in \Cref{thm:totally geodesic lag.} this implies $0 = \sd h^1+\omega_J = -3m^3\wedge m^6$ or equivalently $\sin(\beta)\cos(\beta)=0$. Thus we can assume $c\neq 0$.
We have $h^1 = l_{11}e^1 + l_{12}e^2$ and $h^2 = l_{21}h^1 + l_{22}h^2$. Taking the exterior differential of these equations and contracting both sides by $e_1\wedge e_3$ yields
\[
0 = \sd h^1(e_1,e_3)  = l_{11} \sd e^1(e_1,e_3) + l_{12}\sd e^2(e_1,e_3) = c(l_{21} \sd e^1(e_1,e_3) + l_{22}\sd e^2(e_1,e_3))
\]
and
\[
-2\sqrt{3}\cos(\beta)\sin(\varphi)\cos(\varphi) = \sd h^2(e_1,e_3) = l_{21}\sd e^1(e_1,e_3) + l_{22} \sd e^2(e_1,e_3).
\]
Consequently, $\cos(\beta)\sin(\varphi)\cos(\varphi) =0$ and the Lagrangian is totally geodesic.

Consider the case $\cos(\beta)^2 - 2\sin(\beta)^2 = 0$. If $\sin(\varphi)^2 -\cos(\varphi)^2 \neq 0$, then from \eqref{eq:thm homogeneous} we get $l_{21} = l_{22} = 0$ and we quickly conclude that $L$ is totally geodesic. Hence, we are left with the case $\sin(\varphi)^2 - \cos(\varphi)^2 = 0$. 
In this case we have $h_{k1}^1 = l_{1k}$, thus for $d= 2\sqrt{3}\sin(\theta)\sin(\beta)\sin(\varphi)\cos(\varphi)$ we find
\begin{align*}
l_{11} &=h_{11}^1 = - h_{22}^1 -h_{33}^1 = 2\sqrt{3}l_{22}\sin(\theta)\sin(\beta)\sin(\varphi)\cos(\varphi) = d\cdot l_{22},\\
l_{12} &=h_{21}^1 = h_{11}^2 = -2\sqrt{3}l_{21}\sin(\theta)\sin(\beta)\sin(\varphi)\cos(\varphi) = -d\cdot l_{21}.
\end{align*}
Consider the case $\sin(\varphi)=\cos(\varphi)=\frac{1}{\sqrt{2}}$, $\sin(\beta)=-\frac{1}{\sqrt{3}}$ and $\cos(\beta) = \frac{\sqrt{2}}{\sqrt{3}}$. A simple computation shows that $\mbox{span}\{\omega(e_1),\omega(e_2),\omega(e_3)\}$ is a subalgebra of $\su(3)$ if and only if $l_{12}=l_{21} =0$ and $l_{11} = \sqrt{2}$. This is the same subalgebra as in \Cref{ex:hom RP3}. All other choices of signs in $\sin(\varphi),\cos(\varphi),\sin(\beta),\cos(\beta)$ yield conjugate solutions. Consequently, the homogeneous Lagrangian submanifold is congruent to $\R P^3\subset F_{1,2}(\C^3)$ from \Cref{ex:hom RP3}.
\end{proof}

\appendix

\section{Canonical connections \label{sec:appendix}}
The canonical connection is used to compute covariant derivatives of the moving frame. Here it is briefly explained how this can be done. The essence is just the expression of a covariant derivative in terms of a principal connection for which the reader can also consult \cite{Kobayashi1963,Kobayashi1969}.
\begin{definition}
	A homogeneous manifold $G/H$ is \emph{reductive} if there exists a complement $\m$ of the Lie algebra $\h$ of the isotropy group $H$ such that $\Ad(H)(\m)\subset \m$. Let $\omega_G$ be the \emph{left invariant Maurer-Cartan form} on $G$. Let $\omega_\h$ and $\omega_\m$ be the $\h$ and $\m$ component of $\omega_G$, respectively. 
\end{definition}
A reductive complement $\m$ can be used to define a principal connection.
\begin{definition}
	The \emph{canonical connection} of a reductive complement $\m$ is the principal connection on the principal $H$-bundle $\pi:G\to G/H$ defined by the horizontal distribution
 	\[
	\mathcal{H}_g := \sd L_g (\m) \subset T_g G,
	\]
	where $L_g:G\to G$ denotes left multiplication by $g\in G$. The corresponding connection $1$-form is given by $\omega_\h:TG \to \h$, i.e. $\mathcal{H}_g = \ker(\omega_\h)_g $.
\end{definition}
The tangent bundle of a homogeneous space is naturally isomorphic to the associated bundle $G\times_H \m$ and there is a natural bijective map between vector fields $X\in \Gamma(TG/H)$ and $H$-equivariant functions $\hat{X}:G\to \m$. For $X\in \Gamma(TG/H)$ let $\tilde{X}\in \Gamma(TG)$ denote a lift, i.e. $\pi_*\tilde{X} = X$. 
For a lift $\tilde{X}$ one has $\tilde{X}_{g h} = \sd R_h \tilde{X}_g$. Thus 
\[
\omega_G(\tilde{X}_{gh}) = \sd L_{gh}^{-1} \tilde{X}_{gh} = \sd L_h^{-1} \sd R_h\sd L_{g}^{-1} \tilde{X}_g = \Ad(h)^{-1} \omega_G(\tilde{X}_g).
\]
The $H$-equivariant function $\hat{X}:G\to \m$ corresponding to $X$ is given by
\[
\hat{X} = \omega_\m(\tilde{X}),
\]
and the right-hand-side is independent of the lift $\tilde{X}$. 

The principal connection $\omega_\h$ defines a covariant derivative on $TG/H$ by
\begin{equation}\label{eq:principal derivative}
\widehat{\nabla_X Y} = \sd_{\tilde{X}} \hat{Y} + \ad(\omega_\h(\tilde{X}))\cdot \hat{Y}: G\to \m.
\end{equation}
\begin{lemma}
	The right-hand-side of \eqref{eq:principal derivative} is independent of the lift $\tilde{X}$ of the vector field $X$.
\end{lemma}
Suppose we are given a submanifold $\Sigma\subset G/H$ which is defined by some immersion $\iota:\Sigma\to G/H$. Let $\tilde{\iota}:\Sigma\to G$ be some local lift, i.e. a local moving frame. 
We now use \eqref{eq:principal derivative} to express the covariant derivative with respect to the canonical connection of a vector field $Y\in \Gamma(T(G/H)|_\Sigma)$ by a tangent vector $X_p \in T_p\Sigma$ in a local moving frame. We have
\begin{align*}
\nabla_{X_p} Y &= \sd_{\tilde{\iota}_*X_p} (\omega_\m(\tilde{\iota}_*Y)) + \ad(\omega_\h(\tilde{\iota}_*X)))\cdot \omega_\m (\tilde{\iota}_*Y)\\
&=\sd_{X_p}\omega_\m(Y) + \ad(\omega_\h(X_p))\cdot \omega_\m(Y), \numberthis \label{eq:moving frame derivative}
\end{align*} 
where the last line is just the notation convention introduced in \Cref{not:pullback forms}, i.e. the pullback $\tilde{\iota}^*\alpha$ of a form by a moving frame is simply denoted as $\alpha$.
This formula gives us a way to express covariant derivatives of the canonical connection in terms of the data describing a moving frame.
\begin{acknowledgements}
	The author is supported by the project 3E160361 of the KU Leuven Research Fund.
\end{acknowledgements}

\bibliographystyle{./abbrv_url}
\bibliography{../../../lagrangian.bib} 

\end{document}